\documentclass[a4paper,11pt]{article}
\usepackage{amsmath,amsthm,amssymb,enumitem,amscd,calc,titlesec}
\usepackage[bookmarks=false]{hyperref}
\usepackage{comment}

\renewcommand{\thesection}{\arabic{section}}
\titleformat{\section}{\Large\bf\boldmath}{\thesection.}{2ex}{}{}
\titlespacing{\section}{0ex}{2ex}{1ex}

\renewcommand{\thesubsection}{\arabic{section}.\arabic{subsection}}
\titleformat{\subsection}{\large\bf\boldmath}{\thesubsection.}{2ex}{}{}
\titlespacing{\section}{0ex}{1.5ex}{0.5ex}

\numberwithin{equation}{section}
{\theoremstyle{definition}\newtheorem{definition}{Definition}[section]

\newtheorem{remark}[definition]{Remark}
\newtheorem{assumptions}[definition]{Assumptions}
\newtheorem{remarkletter}{Remark}
}

\newtheorem{lemma}[definition]{Lemma}
\newtheorem{theorem}[definition]{Theorem}

\newtheorem{theoremletter}[remarkletter]{Theorem}

\newtheorem{corollary}[definition]{Corollary}

\newcommand{\C}{\mathbb{C}}
\newcommand{\F}{\mathbb{F}}

\newcommand{\actson}{\curvearrowright}

\newcommand{\Aut}{\operatorname{Aut}}

\newcommand{\T}{\mathbb{T}}
\newcommand{\Z}{\mathbb{Z}}
\newcommand{\cF}{\mathcal{F}}

\newcommand{\id}{\mathord{\operatorname{id}}}

\newcommand{\recht}{\rightarrow}
\newcommand{\cU}{\mathcal{U}}
\newcommand{\vphi}{\varphi}

\newcommand{\R}{\mathbb{R}}
\newcommand{\al}{\alpha}
\newcommand{\eps}{\varepsilon}
\newcommand{\Tr}{\operatorname{Tr}}
\newcommand{\ovt}{\mathbin{\overline{\otimes}}}

\newcommand{\om}{\omega}

\newcommand{\cZ}{\mathcal{Z}}

\newcommand{\ot}{\otimes}

\newcommand{\Ad}{\operatorname{Ad}}
\newcommand{\cG}{\mathcal{G}}

\newcommand{\dpr}{^{\prime\prime}}
\newcommand{\be}{\beta}

\newcommand{\Mtil}{\widetilde{M}}
\newcommand{\Ntil}{\widetilde{N}}

\newcommand{\Stab}{\operatorname{Stab}}

\newcommand{\PSL}{\operatorname{PSL}}

\newcommand{\cS}{\mathcal{S}}

\newcommand{\Om}{\Omega}

\newcommand{\Centr}{\operatorname{Centr}}
\newcommand{\module}{\operatorname{mod}}
\newcommand{\Btil}{\widetilde{B}}
\newcommand{\Atil}{\widetilde{A}}
\newcommand{\Prob}{\operatorname{Prob}}
\newcommand{\Norm}{\operatorname{Norm}}

\begin{document}

\begin{center}
{\boldmath\Large\bf  Families of hyperfinite subfactors with the same standard \vspace{0.5ex}\\ invariant and prescribed fundamental group}
\bigskip

{\sc by Arnaud Brothier\footnote{KU~Leuven, Department of Mathematics, Leuven (Belgium), arnaud.brothier@wis.kuleuven.be \\
    Supported by ERC Starting Grant VNALG-200749} and Stefaan Vaes\footnote{KU~Leuven, Department of Mathematics, Leuven (Belgium), stefaan.vaes@wis.kuleuven.be \\
    Supported by ERC Starting Grant VNALG-200749, Research Programme G.0639.11 of the Research Foundation~-- Flanders (FWO) and KU~Leuven BOF research grant OT/13/079.}}
\end{center}

\begin{abstract}\noindent
We construct irreducible hyperfinite subfactors of index 6 with a prescribed fundamental group from a large family containing all countable and many uncountable subgroups of $\R_+$. We also prove that there are unclassifiably many irreducible hyperfinite group-type subfactors of index 6 that all have the same standard invariant. More precisely, we associate such a subfactor to every ergodic measure preserving automorphism of the interval $[0,1]$ and prove that the resulting subfactors are isomorphic if and only if the automorphisms are conjugate.
\end{abstract}

\section{Introduction and statement of the main results}

To every inclusion of II$_1$ factors $N \subset M$ with finite Jones index \cite{Jo82a} is associated a group-like object $\cG_{N \subset M}$ called the standard invariant. In \cite{Po92}, Popa proved the fundamental result that every strongly amenable standard invariant arises from precisely one hyperfinite subfactor. When the standard invariant is nonamenable, much less is known. It is for instance wide open to decide at which index values larger than $4$, the $A_\infty$ Temperley-Lieb standard invariant arises from a hyperfinite subfactor, and if it does, whether this subfactor is unique or not.

In \cite{BH95}, Bisch and Haagerup associated to every countable group $\Gamma$ generated by finite subgroups $H,K \subset \Gamma$ and to every outer action $(\al_g)_{g \in \Gamma}$ of $\Gamma$ on the hyperfinite II$_1$ factor $R$, the \emph{group-type subfactor} $S(\al) : R^H \subset R \rtimes K$. This construction gives rise to a wealth of infinite depth subfactors with different types of properties (amenable vs.\ strongly amenable, property~(T), etc). Popa proved in \cite{Po01a} a deep cocycle superrigidity theorem for Connes-St{\o}rmer Bernoulli actions of infinite property (T) groups $\Gamma$ and used this to show that all these groups $\Gamma$ admit uncountably many non outer conjugate actions $(\al_g)_{g \in \Gamma}$. This result was then applied in \cite{BNP06} to property (T) groups $\Gamma$ generated by subgroups $H \cong \Z/2\Z$ and $K \cong \Z/3\Z$ and implies that the resulting subfactors $S(\al)$ are nonisomorphic, but nevertheless all have the same standard invariant.

%

Also amplifications can give rise to nonisomorphic subfactors with the same standard invariant. If $N \subset M$ is a subfactor and $t > 0$, the \emph{amplification} $(N \subset M)^t$ is defined as follows~: choose a projection $p \in M_n(\C) \ot N$ with $(\Tr \ot \tau)(p) = t$ and define $(N \subset M)^t$ as the inclusion $p(M_n(\C) \ot N)p \subset p(M_n(\C) \ot M)p$. Following \cite[Definition 5.4.7]{Po87}, the \emph{relative fundamental group} of the subfactor $N \subset M$ is then defined as
$$\cF(N \subset M) = \{t > 0 \mid (N \subset M)^t \cong (N \subset M)\}$$
and is a subgroup of $\R_+$. For the group-type subfactors $S(\al) : R^H \subset R \rtimes K$, it is shown in \cite{BNP06} that $\cF(N \subset M)$ is a subgroup of the fundamental group $\cF(\al)$ introduced in \cite{Po01a}. Since it is proven in \cite{Po01a} that the noncommutative Bernoulli actions of an infinite property (T) group have trivial fundamental group, the resulting subfactors $S(\al)$ also have trivial relative fundamental group and the amplifications $S(\al)^t$, $t > 0$, form an uncountable family of nonisomorphic subfactors with the same standard invariant.

In Theorem \ref{thm.A} below, we refine the above results and show that there are ``unclassifiably'' many nonisomorphic subfactors of index 6 with the same standard invariant. More precisely, to every ergodic measure preserving automorphism $\Delta$ of the interval $[0,1]$, we associate an outer action $\al^\Delta$ of the modular group $\PSL(2,\Z) = \Z/2\Z * \Z/3\Z$ on the hyperfinite II$_1$ factor $R$ and consider the corresponding subfactor $S(\al^\Delta) : R^{\Z/2\Z} \subset R \rtimes \Z/3\Z$. All these subfactors $S(\al^\Delta)$ have index 6 and the same standard invariant $\cG$. We prove that the subfactor $S(\al^\Delta)$ is isomorphic with $S(\al^{\Delta'})$ if and only if $\Delta$ is conjugate to $\Delta'$, meaning that $\Delta' = \theta \circ \Delta \circ \theta^{-1}$ for some measure preserving transformation $\theta$. Since the classification of ergodic transformations up to conjugacy is wild in any possible sense (see e.g.\ \cite{Hj01,FRW08}), the classification of hyperfinite index 6 subfactors with standard invariant $\cG$ is at least as wild.

In Theorem \ref{thm.A}, we also construct outer actions $\al$ of the modular group $\PSL(2,\Z)$ such that the resulting subfactor $S(\al)$ has any prescribed relative fundamental group from the large family $\cS$ of subgroups of $\R_+$ studied in \cite[Section 2]{PV08}. This family $\cS$ contains all countable subgroups of $\R_+$, as well as many uncountable subgroups that can have any Hausdorff dimension between 0 and 1.

Note that the main result of \cite{PV08} showed that all groups in the family $\cS$ arise as the fundamental group $\cF(M)$ of a II$_1$ factor $M$ with separable predual. The result in \cite{PV08} is an existence theorem that ultimately relies on a Baire category argument. Explicit examples of II$_1$ factors with prescribed fundamental group in $\cS$ were constructed in \cite{De10}. In Corollary \ref{cor.fund-group-II1} below, we also give a new and explicit proof of that result, using Theorem \ref{thm.A} and the main results of \cite{PV11,PV12}.

\begin{theoremletter}\label{thm.A}
Let $n \geq 2$ be an integer and $m \geq 3$ a prime number. Put $\Gamma = (\Z / n \Z) * (\Z / m\Z)$. For every outer action $(\al_g)_{g \in \Gamma}$ of $\Gamma$ on the hyperfinite II$_1$ factor, consider the associated group-type subfactor $S(\al) : R^{\Z/n\Z} \subset R \rtimes \Z/m\Z$. Note that these subfactors are irreducible, have index $nm$ and have a standard invariant that only depends on the integers $n$ and $m$.
\begin{enumerate}
\item For every group $H \in \cS$, there exists an outer action $(\al_g)_{g \in \Gamma}$ on the hyperfinite II$_1$ factor $R$ such that the subfactor $S(\al)$ has relative fundamental group $H$.
\item To every ergodic probability measure preserving automorphism $\Delta$ of a standard nonatomic probability space, we can associate an outer action $(\al^\Delta_g)_{g \in \Gamma}$ on $R$ such that the corresponding subfactors $S(\al^\Delta)$ are isomorphic if and only if the automorphisms are conjugate.
\end{enumerate}
\end{theoremletter}

As mentioned above, classifying group-type subfactors $R^H \subset R \rtimes K$ is closely related to classifying actions up to outer/cocycle conjugacy. Two outer actions $\al$ and $\beta$ of a countable group $\Gamma$ on a II$_1$ factor $M$ are called \emph{outer conjugate} if there exist automorphisms $\psi \in \Aut(M)$, $\delta \in \Aut(\Gamma)$ and a family of unitaries $(w_g)_{g \in \Gamma}$ in $M$ such that $\beta_{\delta(g)} \circ \psi = \psi \circ (\Ad w_g) \circ \al_g$ for all $g \in \Gamma$. If the unitaries $w_g$ can be chosen in such a way that $w_{gh} = w_g \, \al_g(w_h)$ for all $g,h \in \Gamma$, then $\al$ and $\beta$ are called \emph{cocycle conjugate}.

There are several parallels between the study of outer actions of $\Gamma$ on the hyperfinite II$_1$ factor $R$ up to cocycle conjugacy and the study of free ergodic probability measure preserving (pmp) actions $\Gamma \actson (X,\mu)$ up to orbit equivalence. Recall for instance that by \cite{OW79} all free ergodic pmp actions of an infinite amenable group $\Gamma$ are orbit equivalent, while it was shown in \cite{Oc85} that all outer actions of an amenable group $\Gamma$ on $R$ are cocycle conjugate.

Nonamenable groups $\Gamma$ admit uncountably many non orbit equivalent actions: this was first proven for the free groups $\F_n$ in \cite{GP03} and then for groups containing a copy of $\F_2$ in \cite{Io07}, and finally in the general case in \cite{Ep07}. This last result is based on \cite{GL07}, where it is shown that every nonamenable group $\Gamma$ contains $\F_2$ ``measurably''. In particular, there is no explicit construction of an uncountable family of non orbit equivalent actions of an arbitrary nonamenable group $\Gamma$, but rather a proof of their existence. Quite surprisingly, our Theorem \ref{thm.B} below provides an explicit and rather easy uncountable family of non outer conjugate actions of an arbitrary nonamenable group $\Gamma$ on the hyperfinite II$_1$ factor. Note here that it was already proven in \cite{Jo82b} that every nonamenable group $\Gamma$ admits at least two non outer conjugate actions on the hyperfinite II$_1$ factor, while it was shown in \cite{Po01a} that every $w$-rigid\footnote{A countable group $\Gamma$ is called $w$-rigid if $\Gamma$ admits an infinite normal subgroup with the relative property (T) of Kazhdan-Margulis.} group $\Gamma$ admits uncountably many non outer conjugate actions.

\begin{theoremletter}\label{thm.B}
Let $\Gamma$ be any nonamenable group and let $\Lambda$ be any amenable group that has a torsion free FC-radical\footnote{The FC-radical of a countable group $\Gamma$ is the normal subgroup that consists of all elements of $\Gamma$ that have a finite conjugacy class.}, e.g.\ an amenable icc group, or an amenable torsion free group. Realize the hyperfinite II$_1$ factor $R$ as
$$R = (M_2(\C)^{\Gamma \times \Lambda} \ovt M_2(\C)^\Lambda) \rtimes \Lambda$$
where $\Lambda$ acts diagonally by Bernoulli shifts and where we take infinite tensor products with respect to the trace on $M_2(\C)$. The Bernoulli shift of $\Gamma$ yields an outer action $(\al^\Lambda_g)_{g \in \Gamma}$ of $\Gamma$ on $R$.

The actions $(\al^{\Lambda_1}_g)_{g \in \Gamma}$ and $(\al^{\Lambda_2}_g)_{g \in \Gamma}$ are outer conjugate if and only if the groups $\Lambda_1,\Lambda_2$ are isomorphic.
\end{theoremletter}

We prove Theorems \ref{thm.A} and \ref{thm.B} by using Popa's deformation/rigidity methods, in particular the spectral gap rigidity of \cite{Po06} and the malleable deformation for Bernoulli actions of \cite{Po01a,Po03}.

Our methods can best be explained for the very explicit actions $(\al^\Lambda_g)_{g \in \Gamma}$ of the arbitrary nonamenable group $\Gamma$ on the hyperfinite II$_1$ factor
$$M(\Lambda)= (M_2(\C)^{\Gamma \times \Lambda} \ovt M_2(\C)^\Lambda) \rtimes \Lambda$$
as defined in Theorem \ref{thm.B}. Contrary to the approach in \cite{Po01a}, we cannot expect to prove a general cocycle superrigidity theorem for $(\al^\Lambda_g)_{g \in \Gamma}$, because $\Gamma$ might be the free group, or a free product group, and such groups do not have cocycle superrigid actions.

So we need to use another method to prove that every outer conjugacy $\psi : M(\Lambda_1) \recht M(\Lambda_2)$ between $(\al^{\Lambda_1}_g)_{g \in \Gamma}$ and $(\al^{\Lambda_2}_g)_{g \in \Gamma}$ is actually a conjugacy up to an inner automorphism. For this, note that by construction, the subalgebra $P_i = M_2(\C)^{\Lambda_i} \rtimes \Lambda_i$ of $M(\Lambda_i)$ is pointwise fixed under the action $(\al^{\Lambda_i}_g)_{g \in \Gamma}$. Using the methods of \cite{Po01a,Po03,Po06}, including a malleable deformation of $M(\Lambda_2)$ and spectral gap rigidity coming from the nonamenability of $\Gamma$, we deduce that the deformation converges uniformly to the identity on the unit ball of $\psi(P_1)$ and find a unitary $w \in M(\Lambda_2)$ such that $w \psi(P_1) w^* = P_2$. So replacing $\psi$ by $(\Ad w) \circ \psi$, the $*$-isomorphism $\psi$ becomes a conjugacy of $(\al^{\Lambda_1}_g)_{g \in \Gamma}$ and $(\al^{\Lambda_2}_g)_{g \in \Gamma}$. Using the mixing techniques of \cite[Section 3]{Po03}, we then finally deduce that $\psi$ must send $M_2(\C)^{\Lambda_1}$ onto $M_2(\C)^{\Lambda_2}$. Since also $\psi(P_1) = P_2$, the actions $\Lambda_i \actson M_2(\C)^{\Lambda_i}$ follow cocycle conjugate and, in particular, $\Lambda_1 \cong \Lambda_2$.

{\bf Acknowledgment.} We are very grateful to Darren Creutz and Cesar E. Silva for their advice on rank one ergodic transformations (see \ref{assum} below).

\section{Preliminaries} \label{sec.prelim}

We start by recalling Popa's theory of \emph{intertwining-by-bimodules} developed in \cite[Section~2]{Po03}. Let $(M,\tau)$ be a von Neumann algebra with separable predual equipped with a normal faithful tracial state. Let $P,Q \subset M$ be von Neumann subalgebras. Following \cite{Po03}, write $P \prec_M Q$ if there exist projections $p \in P$, $q \in Q$, a normal unital $*$-homomorphism $\theta : pPp \recht qQq$ and a nonzero partial isometry $v \in pMq$ satisfying $a v = v \theta(a)$ for all $a \in P$. By \cite[Corollary~2.3]{Po03}, we have $P \not\prec_M Q$ if and only if there exists a sequence of unitaries $v_n \in \cU(P)$ satisfying $\lim_n \|E_Q(a v_n b)\|_2 = 0$ for all $a,b \in M$.

Also recall that a trace preserving action $(\gamma_s)_{s \in \Lambda}$ of a countable group $\Lambda$ on a von Neumann algebra $(B,\tau)$ with normal faithful tracial state $\tau$ is called \emph{mixing} if for all $a,b \in B$ with $\tau(a) = 0 = \tau(b)$, we have $\lim_{s \recht \infty} \tau(\gamma_s(a)b) = 0$.

Our first lemma provides a variant of the results in \cite[Section 3]{Po03}. For completeness, we provide a complete proof.

\begin{lemma} \label{lem.mixing}
Let $(B,\tau)$ and $(D,\Tr)$ be von Neumann algebras equipped with normal faithful traces with $\tau(1) = 1$ and with $\Tr$ being finite or semifinite. Assume that a countable group $\Lambda$ acts in a trace preserving way on $(B,\tau)$ and $(D,\Tr)$. Denote these actions, as well as their diagonal product on $B \ovt D$, by $(\gamma_s)_{s \in \Lambda}$. Assume that the action $\Lambda \actson (B,\tau)$ is mixing.

Let $p \in D$ be a projection with $\Tr(p) < \infty$. Denote $M = p((B \ovt D) \rtimes \Lambda) p$ and $P = p(D \rtimes \Lambda)p$. If $Q \subset P$ is a von Neumann subalgebra with $Q \not\prec_P pDp$ and if $v \in M$ satisfies $v Q \subset P v$, then $v \in P$.
\end{lemma}
\begin{proof}
By assumption, we get a sequence $w_n \in \cU(Q)$ satisfying $\lim_n \|E_{pDp}(x w_n y)\|_2 = 0$ for all $x,y \in P$. Denote by $(u_s)_{s \in \Lambda}$ the canonical unitaries in the crossed product $D \rtimes \Lambda$. Every $w_n$ has a Fourier decomposition
$$w_n = \sum_{s \in \Lambda} (w_n)_s u_s \quad\text{with}\;\; (w_n)_s \in p D \gamma_s(p) \; .$$
We claim that for every fixed $s \in \Lambda$, we have $\lim_n \|(w_n)_s\|_2 = 0$. To prove this claim, fix $s \in \Lambda$. For every unitary $v \in \cU(D)$, we have the element $p u_s^* v p \in P$ and therefore
$$\lim_n \|E_{pDp}(w_n \, p u_s^* v p)\|_2 = 0 \; .$$
Since $E_{pDp}(w_n \, p u_s^* v p) = (w_n)_s \gamma_s(p) v p = (w_n)_s v p$, we get that $\lim_n \|(w_n)_s v p\|_2 = 0$. Then also $\lim_n \|(w_n)_s \, vpv^*\|_2 = 0$. Since the join of all the projections $vpv^*$, $v \in \cU(D)$, equals the central support $z \in \cZ(D)$ of $p \in D$, we get that $\lim_n \|(w_n)_s \, z\|_2 = 0$. But $(w_n)_s \, z = z \, (w_n)_s = (w_n)_s$ and the claim is proven.

We next prove that
\begin{equation}\label{eq.right-mixing}
\lim_n \|E_P(x w_n y)\|_2 = 0 \quad\text{for all}\;\; x,y \in M \ominus P \; .
\end{equation}
Since the linear span of $P \, (B \ominus \C 1)$ is $\|\,\cdot\,\|_2$-dense in $M \ominus P$, it suffices to prove \eqref{eq.right-mixing} for $x,y \in B \ominus \C 1$. But then
$$E_P(x w_n y) = \sum_{s \in \Lambda} \tau(x \gamma_s(y)) \, (w_n)_s \, u_s \; .$$
It follows that
$$\|E_P(x w_n y)\|_2^2 = \sum_{s \in \Lambda} |\tau(x \gamma_s(y))|^2 \, \|(w_n)_s\|_2^2 \; .$$
Fix $\eps > 0$. Since the action $\Lambda \actson (B,\tau)$ is mixing, take a finite subset $\cF \subset \Lambda$ such that $|\tau(x \gamma_s(y))|^2  < \eps / \Tr(p)$ for every $s \in \Lambda- \cF$. Since $\lim_n \|(w_n)_s\|_2 = 0$ for every fixed $s \in \Lambda$, we next take $n_0$ such that
$$\sum_{s \in \cF} |\tau(x \gamma_s(y))|^2 \, \|(w_n)_s\|_2^2 < \eps \quad\text{for all $n \geq n_0$.}$$
We conclude that for all $n \geq n_0$,
\begin{align*}
\|E_P(x w_n y)\|_2^2 & \leq \eps + \sum_{s \in \Lambda - \cF} |\tau(x \gamma_s(y))|^2 \, \|(w_n)_s\|_2^2 \\ & \leq \eps + \frac{\eps}{\Tr(p)} \sum_{s \in \Lambda-\cF} \|(w_n)\|_2^2 \leq \eps + \frac{\eps}{\Tr(p)} \, \|w_n\|_2^2 = 2 \eps \; .
\end{align*}
So \eqref{eq.right-mixing} is proven. The conclusion of the lemma now follows from \cite[Lemma D.3]{Va06}.
\end{proof}

Let $(M,\tau)$ be a tracial von Neumann algebra. Recall from \cite[Section 1.2]{Po81} that von Neumann subalgebras $M_1, M_2 \subset M$ are said to form a \emph{commuting square} when $E_{M_1} \circ E_{M_2} = E_{M_1 \cap M_2} = E_{M_2} \circ E_{M_1}$. We need the following easy lemma and include a complete proof for the convenience of the reader.

\begin{lemma}\label{lem.commuting}
Let $(M,\tau)$ be a tracial von Neumann algebra with von Neumann subalgebras $M_1, M_2 \subset M$ that form a commuting square. Assume that the linear span of $M_1 M_2$ is $\|\,\cdot\,\|_2$-dense in $M$. If $Q \subset M_1$ is a von Neumann subalgebra and $Q \not\prec_{M_1} M_1 \cap M_2$, then $Q \not\prec_M M_2$.
\end{lemma}
\begin{proof}
Put $P = M_1 \cap M_2$. Assume that $Q \subset M_1$ and $Q \not\prec_{M_1} P$. We then find a sequence of unitaries $w_n \in \cU(Q)$ satisfying $\lim_n \|E_P(x w_n y)\|_2 = 0$ for all $x,y \in M_1$. For all $x,y \in M_1$ and for all $a,b \in M_2$, we have
$$E_{M_2}(ax \, w_n \, yb) = a \, E_{M_2}(x w_n y) \, b = a \, E_P(x w_n y) \, b \; .$$
Therefore, $\lim_n \|E_{M_2}(ax \, w_n \, y b)\|_2 = 0$. Since the linear span of $M_1 M_2$ is $\|\,\cdot\,\|_2$-dense in $M$, we conclude that $\lim_n \|E_{M_2}(c w_n d)\|_2 = 0$ for all $c,d \in M$. This implies that $Q \not\prec_M M_2$.
\end{proof}

We finally recall the concept of a \emph{co-induced action.} Assume that $(B,\tau)$ is a tracial von Neumann algebra with a trace preserving action $(\beta_g)_{g \in \Gamma_0}$. Assume that $\Gamma_0 < \Gamma$. The co-induced action of $(\beta_g)_{g \in \Gamma_0}$ to $\Gamma$ is the following action $(\al_g)_{g \in \Gamma}$ on the infinite tensor product $(A,\tau) = (B,\tau)^{\Gamma/\Gamma_0}$. First choose a section $\theta : \Gamma / \Gamma_0 \recht \Gamma$ with $\theta(e\Gamma_0) = e$. We then get the $1$-cocycle $\om : \Gamma \times \Gamma/\Gamma_0 \recht \Gamma_0$ determined by
$$g \, \theta(h \Gamma_0) = \theta(gh \Gamma_0) \, \om(g,h\Gamma_0) \quad\text{for all}\;\; g \in \Gamma \; , \; h\Gamma_0 \in \Gamma/\Gamma_0 \; .$$
We denote by $\pi_{h \Gamma_0} : B \recht A$ the embedding of $B$ as the $h\Gamma_0$-th tensor factor. There is a unique trace preserving action $(\al_g)_{g \in \Gamma}$ of $\Gamma$ on $A$ satisfying
$$\al_g(\pi_{h\Gamma_0}(b)) = \pi_{gh \Gamma_0}(\beta_{\om(g,h\Gamma_0)}(b)) \quad\text{for all}\;\; g \in \Gamma \; , \; h\Gamma_0 \in \Gamma/\Gamma_0 \; , \; b \in B \; .$$
Note that by construction $\al_g(\pi_{e\Gamma_0}(b)) = \pi_{e\Gamma_0}(\beta_g(b))$ for all $g \in \Gamma_0$, $b \in B$.

\section{A first outer conjugacy lemma and the proof of Theorem \ref{thm.B}}\label{sec.outer-conj}

Throughout this section, we fix a countable group $\Gamma$ with a subgroup $\Gamma_0 < \Gamma$ that is \emph{not co-amenable,} i.e.\ such that the set $\Gamma/\Gamma_0$ does not admit a $\Gamma$-invariant mean. In particular, one can take $\Gamma_0=\{e\}$, or $\Gamma_0$ amenable, and $\Gamma$ any nonamenable group. We also fix an infinite group $\Lambda$. We let these groups $\Gamma$ and $\Lambda$ act in the following way on von Neumann algebras.
\begin{itemize}
\item Let $(B,\tau)$ be a von Neumann algebra equipped with a normal faithful tracial state $\tau$. We assume that $(B,\tau)$ comes with commuting trace preserving faithful\footnote{We call a group action \emph{faithful} if no non trivial group element acts by the identity automorphism.} actions $(\beta_g)_{g \in \Gamma_0}$ and $(\gamma_s)_{s \in \Lambda}$. We assume that the action $(\gamma_s)_{s \in \Lambda}$ on $(B,\tau)$ is \emph{mixing.}
\item We put $(A,\tau) = (B,\tau)^{\Gamma/\Gamma_0}$. As recalled at the end of Section \ref{sec.prelim}, we can define the co-induced action of $(\beta_g)_{g \in \Gamma_0}$ and this is an action $(\al_g)_{g \in \Gamma}$ of $\Gamma$ on $(A,\tau)$. We also consider the diagonal action of $\Lambda$ on $(A,\tau)$ that we still denote as $(\gamma_s)_{s \in \Lambda}$. Note that $(\al_g)_{g \in \Gamma}$ commutes with $(\gamma_s)_{s \in \Lambda}$.
\item Let $(D,\Tr)$ be a von Neumann algebra equipped with a normal, finite or semifinite, faithful trace $\Tr$. Assume that $(\gamma_s)_{s \in \Lambda}$ is a trace preserving action on $(D,\Tr)$ and that one of the following assumptions hold.
    \begin{enumerate}
    \item $D$ is a factor, the action $(\gamma_s)_{s \in \Lambda}$ is outer and the group $\Lambda$ has a torsion free FC-radical, e.g.\ an icc group, or a torsion free group.
    \item $D$ is diffuse abelian and the action $(\gamma_s)_{s \in \Lambda}$ is essentially free and ergodic.
    \end{enumerate}
We consider the diagonal action of $\Lambda$ on $A \ovt D$ and continue to denote all these actions of $\Lambda$ by $(\gamma_s)_{s \in \Lambda}$.
\end{itemize}
These data yield the crossed product von Neumann algebra $N = (A \ovt D) \rtimes \Lambda$ equipped with the trace $\Tr$ induced by $\tau$ and $\Tr$. The action $(\al_g)_{g \in \Gamma}$ of $\Gamma$ on $A$ extends to an action on $N$ that equals the identity on $D \rtimes \Lambda$ and that we still denote as $(\al_g)_{g \in \Gamma}$. We start by proving a few basic properties.

\begin{lemma}\label{lem1}
The von Neumann algebra $N$ is a factor. We have $N \cap (D \rtimes \Lambda)' = \C 1$. The action $(\al_g)_{g \in \Gamma}$ of $\Gamma$ on $N$ is outer.
\end{lemma}
\begin{proof}
We first prove that $N \cap (D \rtimes \Lambda)' = \C 1$. In the case where $D$ is a factor and the action of $\Lambda$ on $D$ is outer, we have that $N \cap (1 \ot D)' = A \ot 1$. Since the action of $\Lambda$ on $(B,\tau)$ is mixing, the diagonal action on $(A,\tau)$ is still mixing, in particular ergodic, so that $N \cap (D \rtimes \Lambda)' = \C 1$.

In the case where $D=L^\infty(Z,\eta)$ is diffuse abelian and the action $\Lambda \actson (Z,\eta)$ is essentially free and ergodic, the essential freeness implies that $N \cap (1 \ot D)' = A \ovt D$. So we must prove that all $\Lambda$-invariant elements in $A \ovt D$ are scalar multiples of $1$. Let $F : Z \recht A \ominus \C 1$ be a measurable function satisfying $F(s \cdot z) = \gamma_s(F(z))$ for all $s \in \Lambda$ and a.e.\ $z \in Z$. We must prove that $F$ is zero a.e. Since $\Lambda$ acts ergodically on $(Z,\eta)$, the map $z \mapsto \|F(z)\|_2$ is constant a.e. If this constant differs from zero, we may assume that it is equal to $1$ a.e. We can then choose $a \in A \ominus \C 1$ with $\|a\|_2 = 1$ such that
$$\cU = \{z \in Z \mid \|F(z)-a\|_2 < 1/3\}$$
is nonnegligible. Since the action of $\Lambda$ on $(A,\tau)$ is mixing, we can take a finite subset $\cF \subset \Lambda$ such that $|\langle \gamma_s(a),a\rangle| < 1/3$ for all $s \in \Lambda - \cF$. We derive as follows that $\eta(s \cdot \cU \cap \cU) = 0$ for all $s \in \Lambda- \cF$. Indeed, otherwise we find $s \in \Lambda-\cF$ and a point $z \in \cU$ such that $s \cdot z \in \cU$, $\|F(z)\|_2=1$ and $\gamma_s(F(z)) = F(s \cdot z)$. But then we arrive at the contradiction
$$1 = |\langle F(s \cdot z),F(s \cdot z)\rangle| = |\langle \gamma_s(F(z)),F(s \cdot z)\rangle| < 2/3 + |\langle \gamma_s(a),a \rangle| < 1 \; .$$
So for almost every $z \in \cU$, we have that $\Lambda \cdot z \cap \cU \subset \cF \cdot z$. Therefore the restriction of the orbit equivalence relation of $\Lambda \actson Z$ to the nonnegligible subset $\cU$ has finite orbits almost everywhere. But this equivalence relation is ergodic and $\cU$ is nonatomic. This is absurd and the conclusion that $N \cap (D \rtimes \Lambda)' = \C 1$ follows.

We have in particular that $N$ is a factor. Assume that $g \in \Gamma$ and $V \in \cU(N)$ with $\al_g = \Ad V$. Since $\al_g(d) = d$ for all $d \in D \rtimes \Lambda$, it follows that $V$ is scalar. Hence $\al_g = \id$. Since $B \neq \C 1$ and since the action $(\beta_g)_{g \in \Gamma_0}$ of $\Gamma_0$ on $B$ is faithful, also the action $(\al_g)_{g \in \Gamma}$ of $\Gamma$ on $A$ is faithful. We conclude that $g = e$.
\end{proof}

We also record the following elementary result that we will need in Section \ref{sec.proofA}.

\begin{lemma}\label{lem.regular}
Let $\Gamma_1 < \Gamma$ be a torsion free subgroup and $\Gamma/\Gamma_0 = I \sqcup J$ a partition of $\Gamma/\Gamma_0$ into $\Gamma_1$-invariant subsets such that $g^{-1} \Gamma_1 g \cap \Gamma_0 = \{e\}$ whenever $g \Gamma_0 \in I$. Define the Hilbert space
$$L := L^2(N) \ominus L^2((B^J \ovt D) \rtimes \Lambda) \; .$$
Then the unitary representation $(\al_g)_{g \in \Gamma_1}$ of $\Gamma_1$ on $L$ is a multiple of the regular representation of $\Gamma_1$.
\end{lemma}
\begin{proof}
For every finite nonempty subset $\cF \subset I$, define $L_\cF \subset L$ as the closed linear span of $((B \ominus \C 1)^\cF \, B^J \ot 1) L^2(D \rtimes \Lambda)$. Note that $L$ is the orthogonal direct sum of all the $L_\cF$. Fix a nonempty finite subset $\cF \subset I$ and define $\Gamma_2 = \{g \in \Gamma_1 \mid g \cF = \cF\}$. Since for every $g \in \Gamma_1$, we have $\al_g(L_\cF) = L_{g \cF}$, it suffices to prove that $\Gamma_2 = \{e\}$. From our assumption that $g^{-1} \Gamma_1 g \cap \Gamma_0 = \{e\}$ whenever $g \Gamma_0 \in I$, it follows that the subgroup $\Gamma_3 = \{g \in \Gamma_1 \mid \forall i \in \cF, g \cdot i = i\}$ equals $\{e\}$. Since $\cF$ is a finite set, $\Gamma_3 < \Gamma_2$ has finite index. So $\Gamma_2$ is finite. Since $\Gamma_1$ is torsion free, it follows that $\Gamma_2 = \{e\}$.
\end{proof}

The aim of this section is to understand when these actions $(\al_g)_{g \in \Gamma}$ are outer conjugate, keeping fixed $\Gamma_0 < \Gamma$ but varying all the other data,

So we keep $\Gamma_0 < \Gamma$ fixed, but further assume that we have, for $i=1,2$, von Neumann algebras $(B_i,\tau)$ and $(D_i,\Tr)$, infinite groups $\Lambda_i$ and actions $(\be^i_g)_{g \in \Gamma_0}$ and $(\gamma^i_s)_{s \in \Lambda_i}$. This results into factors $N_i$ with outer actions $(\al^i_g)_{g \in \Gamma}$.

Throughout, we keep as standing assumptions the properties listed in the beginning of this section.

\begin{lemma}\label{lem2}
Assume that $\psi : N_1 \recht N_2$ is an outer conjugacy between $(\al^1_g)_{g \in \Gamma}$ and $(\al^2_g)_{g \in \Gamma}$. Then there exists a unitary $w \in \cU(N_2)$ and an automorphism $\delta \in \Aut(\Gamma)$ such that the isomorphism $\psi' = (\Ad w) \circ \psi$ satisfies
$$\psi'(D_1 \rtimes \Lambda_1) = D_2 \rtimes \Lambda_2 \;\; , \quad \psi'(D_1) = D_2 \quad\text{and}\quad \psi' \circ \al^1_g = \al^2_{\delta(g)} \circ \psi' \quad\text{for all}\;\; g \in \Gamma \; .$$
\end{lemma}

Our proof of Lemma \ref{lem2} is very similar to the proof of \cite[Theorem 4.1]{Po06}. We use the spectral gap methods of \cite{Po06} and the malleable deformation for co-induced actions developed in \cite{Po01a,Po03}. We more precisely use the following variant of that malleable deformation, due to \cite{Io06}. To introduce the notations, we drop the indices $i=1,2$ from $B_i$, $A_i$, $D_i$, etc.

Define $\Btil = B * L \Z$ with respect to the natural tracial states that we all denote by $\tau$. Denote by $(u_n)_{n \in \Z}$ the canonical unitaries in $L \Z$ and define $h \in L \Z$ as the selfadjoint element with spectrum $[-\pi,\pi]$ satisfying $u_1 = \exp(ih)$. For every $t \in \R$, we put $u_t = \exp(ith)$ and we define the $1$-parameter group $(\zeta_t)_{t \in \R}$ of inner automorphisms of $\Btil$ given by $\zeta_t = \Ad u_t$. We extend the actions $(\beta_g)_{g \in \Gamma_0}$ and $(\gamma_s)_{s \in \Lambda}$ to $\Btil$ by acting trivially on $L \Z$. These two actions and the action $(\zeta_t)_{t \in \R}$ all commute.

Define $\Atil$ as the infinite tensor product $\Atil = \Btil^{\Gamma/\Gamma_0}$. We continue to denote by $(\gamma_s)_{s \in \Lambda}$ and $(\zeta_t)_{t \in \R}$ the diagonal actions on $\Atil$. They commute with the co-induced action $(\al_g)_{g \in \Gamma}$ on $\Atil$. Moreover this co-induced action extends the action $(\al_g)_{g \in \Gamma}$ on $A$. We finally consider the crossed product $\Ntil = (\Atil \ovt D) \rtimes \Lambda$ with respect to the diagonal action of $\Lambda$, together with the action $(\al_g)_{g \in \Gamma}$ of $\Gamma$ on $\Ntil$ that extends the given action on $\Atil$ and that is the identity on $D \rtimes \Lambda$.

Our assumption that $\Gamma_0$ is not co-amenable in $\Gamma$ is used to obtain the following result.

\begin{lemma}\label{lem.gap}
Assume that $(V_g)_{g \in \Gamma}$ are unitaries in $\cU(N)$ and that $\Om : \Gamma \times \Gamma \recht \T$ is a map satisfying $V_g \, \al_g(V_h) = \Om(g,h) \, V_{gh}$ for all $g,h \in \Gamma$. The unitary representation
$$\rho : \Gamma \recht \cU(L^2(\Ntil \ominus N)) : \rho_g(\xi) = V_g \, \al_g(\xi) \, V_g^*$$
does not weakly contain the trivial representation.
\end{lemma}
\begin{proof}
We write $H = L^2(\Ntil \ominus N)$. Denote by $\cS$ the set of all finite nonempty subsets $\cF \subset \Gamma/\Gamma_0$. For every $\cF \in \cS$, we define $H(\cF)$ as the closed linear span of $ ((\Btil \ominus B)^\cF \ot 1) L^2(N)$ inside $L^2(\Ntil \ominus N)$. One checks that $H$ is the orthogonal direct sum of the subspaces $H(\cF)$, $\cF \in \cS$. We have that $\rho_g(H(\cF)) = H(g \cF)$ for all $g \in \Gamma$ and $\cF \in \cS$. Also note that $H(\cF)$ is an $N$-$N$-subbimodule of $L^2(\Ntil \ominus N)$.

Denote by $\Prob(\cS)$ the set of probability measures on the countable set $\cS$. Denote by $(H)_1$ the set of unit vectors in $H$. The map
$$\theta : (H)_1 \recht \Prob(\cS) : (\theta(\xi))(\cF) = \|P_{H(\cF)}(\xi)\|_2^2$$
satisfies $\theta(\rho_g(\xi)) = g \cdot \theta(\xi)$.

Assume now that the unitary representation $\rho$ weakly contains the trivial representation, i.e.\ admits a sequence of unit vectors $\xi_n \in (H)_1$ satisfying $\lim_n \|\rho_g(\xi_n) - \xi_n\|_2 = 0$ for all $g \in \Gamma$. We will prove that $\Gamma_0$ is co-amenable in $\Gamma$. Define $\om_n = \theta(\xi_n)$. Then $\om_n$ is a sequence of probability measures on $\cS$ satisfying $\lim_n \|g \cdot \om_n - \om_n\|_1 = 0$.

Choose a set $\cS_0 \subset \cS$ of representatives for the orbits of the action $\Gamma \actson \cS$. We make this choice such that $e\Gamma_0 \in \cF$ for every $\cF \in \cS_0$. For every $\cF \in \cS_0$, define $\Norm(\cF) = \{g \in \Gamma \mid g \cF = \cF\}$. Since $\cS_0$ is a set of representatives for the action $\Gamma \actson \cS$, we identify $\cS$ with the disjoint union of the sets $\Gamma/\Norm(\cF)$, $\cF \in \cS_0$.

For every $\cF \in \cS_0$, write $\Stab(\cF) = \{g \in \Gamma \mid g h \Gamma_0 = h \Gamma_0 \;\;\text{for all}\;\; h\Gamma_0 \in \cF\}$. Since all the $\cF$ are finite sets, we have that $\Stab(\cF)$ is a finite index subgroup of $\Norm(\cF)$. We define $\cS'$ as the disjoint union of the sets $\Gamma / \Stab(\cF)$, $\cF \in \cS_0$. Putting together the finite-to-one maps $\Gamma/\Stab(\cF) \recht \Gamma/\Norm(\cF)$, we obtain the $\Gamma$-equivariant finite-to-one map $\theta' : \cS' \recht \cS$. This map $\theta'$ induces a $\Gamma$-equivariant isometry of $\Prob(\cS)$ into $\Prob(\cS')$. Applying this isometry to $\om_n$, we find a sequence of probability measures $\om'_n$ on $\cS'$ satisfying $\lim_n \|g \cdot \om'_n - \om'_n\|_1 = 0$ for all $g \in \Gamma$. Taking a weak$^*$ limit point, it follows that the action $\Gamma \actson \cS'$ admits an invariant mean. For every $\cF \in \cS_0$, we have $e \Gamma_0 \in \cF$ and therefore $\Stab(\cF) \subset \Gamma_0$. Define the map $\theta\dpr : \cS' \recht \Gamma / \Gamma_0$ given by $\theta\dpr(h \Stab(\cF)) = h \Gamma_0$ for all $\cF \in \cS_0$ and $h \in \Gamma$. Since $\theta\dpr$ is $\Gamma$-equivariant, we push forward the $\Gamma$-invariant mean on $\cS'$ to a $\Gamma$-invariant mean on $\Gamma/\Gamma_0$. This precisely means that $\Gamma_0$ is co-amenable inside $\Gamma$.
\end{proof}

We are now ready to prove Lemma \ref{lem2}.

\begin{proof}[Proof of Lemma \ref{lem2}]
Fix a projection $p_1 \in D_1 \subset N_1$ with $0 < \Tr(p_1) < \infty$. After unitarily conjugating $\psi$, we may assume that $\psi(p_1) \in D_2$. Put $p_2 = \psi(p_1)$.

Since $\psi$ is an outer conjugacy between $(\al^1_g)_{g \in \Gamma}$ and $(\al^2_g)_{g \in \Gamma}$, we find an automorphism $\delta \in \Aut(\Gamma)$, unitaries $(V_g)_{g \in \Gamma}$ in $\cU(N_2)$ and a map $\Om : \Gamma \times \Gamma \recht \T$ such that
$$\psi \circ \al^1_{\delta^{-1}(g)} = (\Ad V_g) \circ \al^2_g \circ \psi \quad\text{and}\quad V_g \, \al^2_g(V_h) = \Om(g,h) \, V_{gh} \quad\text{for all}\;\; g,h \in \Gamma \; .$$
Define as above the malleable deformation $(\zeta_t)_{t \in \R}$ of $\Ntil_2 = (\Atil_2 \ovt D_2) \rtimes \Lambda_2$. Denote by $\rho$ the unitary representation of $\Gamma$ on $L^2(\Ntil_2 \ominus N_2)$ given by $\rho_g(\xi) = V_g \, \al^2_g(\xi) \, V_g^*$. By Lemma \ref{lem.gap}, $\rho$ does not weakly contain the trivial representation. We then find a constant $\kappa > 0$ and a finite subset $\cF \subset \Gamma$ such that
\begin{equation}\label{eq.mygap}
\|\xi\|_2 \leq \kappa \sum_{g \in \cF} \|\rho_g(\xi) - \xi\|_2 \quad\text{for all}\;\; \xi \in L^2(\Ntil_2 \ominus N_2) \; .
\end{equation}
We write $M_i = p_i N_i p_i$ and $\Mtil_i = p_i \Ntil_i p_i$. Put $P_i = p_i (D_i \rtimes \Lambda_i) p_i$. Put $\eps = \|p_2\|_2/4$ and $\delta = \eps / (2 \kappa \, |\cF|)$. Take an integer $n_0$ large enough such that $t=n_0^{-1}$ satisfies
$$\|(V_g - \zeta_t(V_g)) p_2\|_2 \leq \delta \quad\text{for all}\;\; g \in \cF \; .$$
For every $a \in P_1$ and $g \in \Gamma$, we have $\al^1_g(a) = a$ and therefore $V_g \al^2_g(\psi(a)) V_g^* = \psi(a)$. By our choice of $t$, we get that
\begin{equation}\label{eq.star}
\|V_g \, \al^2_g (\zeta_t(b)) \, V_g^* - \zeta_t(b)\|_2 \leq 2\delta \quad\text{for all $g \in \cF$ and all $b \in \psi(P_1)$ with $\|b\| \leq 1$.}
\end{equation}
Denote by $E : \Ntil \recht N$ the unique trace preserving conditional expectation. Whenever $b \in \psi(P_1)$ with $\|b\|\leq 1$, we put $\xi = \zeta_t(b) - E(\zeta_t(b))$ and conclude from \eqref{eq.star} that
$$\kappa \sum_{g \in \cF} \|\rho_g(\xi) - \xi\|_2 \leq 2 \kappa \, |\cF|\, \delta = \eps \; .$$
It follows from \eqref{eq.mygap} that $\|\xi\|_2 \leq \eps$. A direct computation shows that $(\zeta_t)$ satisfies the following transversality property of \cite[Lemma 2.1]{Po06}.
$$\|b - \zeta_t(b)\|_2 \leq \sqrt{2} \, \|\zeta_t(b) - E(\zeta_t(b))\|_2 \quad\text{for all}\;\; b \in M_2 \; .$$
We conclude that $\|b - \zeta_t(b)\|_2 \leq 2 \eps$ for all $b \in \psi(P_1)$ with $\|b\|\leq 1$. It follows that for all $b \in \cU(\psi(P_1))$,
$$|\Tr(b \zeta_t(b^*)) - \Tr(b b^*)| \leq \|b\|_2 \, \|b - \zeta_t(b)\|_2 \leq \|p_2\|_2 \, 2\eps = \Tr(p_2)/2 \; .$$
So $\Tr(b \zeta_t(b^*)) \geq \Tr(p_2)/2$ for all $b \in \cU(\psi(P_1))$.

Defining $W \in \Mtil_2$ as the unique element of minimal $\|\,\cdot\,\|_2$ in the weakly closed convex hull of $\{b \zeta_t(b^*) \mid b \in \cU(\psi(P_1))\}$, it follows that $\Tr(W) \geq \Tr(p_2)/2$ and $b W = W \zeta_t(b)$ for all $b \in \psi(P_1)$. In particular, $W$ is a nonzero element of $\Mtil_2$ and $WW^*$ commutes with $\psi(P_1)$.

Since $V_g \al^2_g(b) V_g^* = b$ for all $g \in \Gamma$ and all $b \in \psi(P_1)$, the elements $W_g := V_g \, \al^2_g(W) \, \zeta_t(V_g^*)$ also satisfy $b W_g = W_g \zeta_t(b)$ for all $b \in \psi(P_1)$. The join of the left support projections of all $W_g$, $g \in \Gamma$, is a projection $q \in \Mtil_2 \cap \psi(P_1)'$ that satisfies $q = V_g \, \al^2_g(q) \, V_g^*$ for all $g \in \Gamma$. By Lemma \ref{lem.gap}, $q \in M_2$. But then, by Lemma \ref{lem1}, we get that $q \in \psi(M_1 \cap P_1') = \C p_2$. Since $q$ is nonzero, we conclude that $q = p_2$. It follows that we can find a $g \in \Gamma$ such that $W' = W \zeta_t(W_g)$ is nonzero. By construction, we have $b W' = W' \zeta_{2t}(b)$ for all $b \in \psi(P_1)$. We can repeat the same reasoning inductively. Since $t = 1/n_0$, we find a nonzero element $W \in \Mtil_2$ satisfying $b W = W \zeta_1(b)$ for all $b \in \psi(P_1)$.

For every finite subset $\cF \subset \Gamma/\Gamma_0$, we define $M_2(\cF) = p_2((B_2^\cF \ovt D_2) \rtimes \Lambda_2)p_2$. We claim that there exists a finite subset $\cF \subset \Gamma/\Gamma_0$ such that $\psi(P_1) \prec_{M_2} M_2(\cF)$. Indeed, if this is not the case, we find a sequence of unitaries $b_n \in \cU(\psi(P_1))$ satisfying
$$\|E_{M_2(\cF)}(x b_n y)\|_2 \recht 0 \quad\text{for all}\;\; x,y \in M_2 \;\;\text{and all finite subsets}\;\; \cF \subset \Gamma/\Gamma_0 \; .$$
We claim that
\begin{equation}\label{eq.claim}
\|E_{M_2}(x \zeta_1(b_n) y)\|_2 \recht 0 \quad\text{for all}\;\; x,y \in \Mtil_2 \; .
\end{equation}
Since the linear span of all $M_2 \Btil_2^\cF$, $\cF \subset \Gamma/\Gamma_0$ finite, is $\|\,\cdot\,\|_2$-dense in $\Mtil_2$, it
suffices to prove \eqref{eq.claim} for all $x,y \in \Btil_2^\cF p_2$ and all finite subsets $\cF \subset \Gamma/\Gamma_0$. But for such $x,y \in \Btil_2^\cF p_2$, we have
$$E_{M_2}(x \zeta_1(b_n) y) = E_{M_2}\bigl(x \zeta_1(E_{M_2(\cF)}(b_n)) y \bigr)$$
and the conclusion follows from our choice of $(b_n)$. So \eqref{eq.claim} is proven. It follows in particular that $\|E_{M_2}(W \zeta_1(b_n) W^*)\|_2 \recht 0$. Since
$$E_{M_2}(W \zeta_1(b_n) W^*) = E_{M_2}(b_n WW^*) = b_n E_{M_2}(WW^*)$$
and since $b_n$ is unitary, we conclude that $WW^*= 0$. This is absurd and we have proven the existence of a finite subset $\cF \subset \Gamma/\Gamma_0$ such that $\psi(P_1) \prec_{M_2} M_2(\cF)$.

Note that $\psi(P_1) \not\prec p_2(B_2^\cF \ovt D_2)p_2$, because otherwise we can take the relative commutant, apply \cite[Lemma 3.5]{Va07} and reach the contradiction that
$$(B_2^{\Gamma/\Gamma_0 - \cF} \ot 1)p_2 \prec M_2 \cap \psi(P_1)' = \psi(M_1 \cap P_1') = \C p_2 \; .$$
In combination with the previous paragraph and \cite[Remark 3.8]{Va07}, we find projections $q_1 \in P_1$ and $q_2 \in p_2 D_2 p_2$, a $*$-homomorphism $\theta : q_1 P_1 q_1 \recht q_2 M_2(\cF) q_2$ and a nonzero partial isometry $V \in \psi(q_1) M_2 q_2$ satisfying $\psi(b) V = V \theta(b)$ for all $b \in q_1 P_1 q_1$, and satisfying $\theta(q_1 P_1 q_1) \not\prec q_2(B_2^\cF \ovt D_2)q_2$.

The projection $VV^*$ commutes with $\psi(q_1 P_1 q_1)$ and hence must be equal to $\psi(q_1)$. The projection $V^* V$ commutes with $\theta(q_1 P_1 q_1)$. Since the action of $\Lambda_2$ on $B_2^{\Gamma/\Gamma_0 - \cF}$ is mixing, it follows from Lemma \ref{lem.mixing} that $V^* V \in q_2 M_2(\cF) q_2$. So we may assume that $V^* V = q_2$. Since $P_1$ and $M_2(\cF)$ are factors, we can then amplify $V$ to a unitary element $V \in \cU(M_2)$ satisfying $V^* \psi(P_1) V \subset M_2(\cF)$.

Since $\Gamma_0$ is not co-amenable inside $\Gamma$, it certainly has infinite index. Therefore we can find $g \in \Gamma$ such that $g \cF \cap \cF = \emptyset$ (see e.g.\ \cite[Lemma 2.4]{PV06}). Denote $Q_2 = V^* \psi(P_1) V$. So $Q_2 \subset M_2(\cF)$. Since $\al^1_g(P_1) = P_1$, it follows that the von Neumann algebras $Q_2$ and $\al^2_g(Q_2)$ are unitarily conjugate inside $M_2$. Since $Q_2 \subset M_2(\cF)$ and $\al^2_g(Q_2) \subset M_2(g \cF)$, it follows from Lemma \ref{lem.commuting} that $Q_2 \prec M_2(\emptyset) = P_2$. Reasoning as above, we find a unitary $V \in \cU(M_2)$ such that $V^* \psi(P_1) V \subset P_2$.

The same reasoning applies to $\psi^{-1}$ and we also find $W \in \cU(M_1)$ such that $W^* \psi^{-1}(P_2) W \subset P_1$. Writing $T = \psi(W) V$, we get that
\begin{equation}\label{eq.inclusions}
T^* P_2 T \subset V^* \psi(P_1) V \subset P_2 \; .
\end{equation}
Since the action of $\Lambda_2$ on $A_2$ is mixing, it follows from Lemma \ref{lem.mixing} that $T \in P_2$. But then the inclusions in \eqref{eq.inclusions} are equalities and we conclude that $V^* \psi(P_1) V = P_2$.

So after a unitary conjugacy of $\psi$, we may from now on assume that $\psi(D_1 \rtimes \Lambda_1) = D_2 \rtimes \Lambda_2$ and $\psi(p_1) = p_2$. Inside $P_2$, we must have the embedding $\psi(p_1 D_1 p_1) \prec p_2 D_2 p_2$. Indeed, since the action $\Lambda_2 \actson A_2$ is mixing and $\psi(A_1 p_1)$ commutes with $\psi(p_1 D_1 p_1)$, it would otherwise follow from Lemma \ref{lem.mixing} that $\psi(A_1 p_1) \subset P_2$ and hence $\psi(M_1) \subset P_2$, which is absurd. We have a similar embedding statement for $\psi^{-1}$.

In the case where the $D_i$ are abelian, $D_i p_i$ is a Cartan subalgebra of $P_i$. It then follows from \cite[Theorem A.1]{Po01b} that $\psi(D_1 p_1)$ can be unitarily conjugated onto $D_2 p_2$ inside $P_2$. In the case where the $D_i$ are factors and the groups $\Lambda_i$ have a torsion free FC-radical, \cite[Lemma 8.4]{IPP05} yields the same conclusion\footnote{The statement of \cite[Lemma 8.4]{IPP05} requires the groups $\Lambda_i$ to be icc, but the proof of \cite[Lemma 8.4]{IPP05} only uses the following property~: if $K < \Lambda_i$ is a finite subgroup and $H < \Lambda_i$ is a finite index subgroup such that $K$ is normal in $H$, then $K =  \{e\}$. This last property is equivalent with the torsion freeness of the FC-radical of $\Lambda_i$.}. So after a further unitary conjugacy of $\psi$, with a unitary from $D_2 \rtimes \Lambda_2$, we arrive at $\psi(D_1 \rtimes \Lambda_1) = D_2 \rtimes \Lambda_2$ and $\psi(D_1) = D_2$.

Using Lemma \ref{lem1}, we then get that $V_g \in N_2 \cap (D_2 \rtimes \Lambda_2)' = \C 1$ and hence, $\psi \circ \al^1_g = \al^2_\delta(g) \circ \psi$ for all $g \in \Gamma$.
\end{proof}

Theorem \ref{thm.B} is an immediate consequence of Lemma \ref{lem2}.

\begin{proof}[Proof of Theorem \ref{thm.B}]
Assume that $\al^{\Lambda_1}$ and $\al^{\Lambda_2}$ are outer conjugate. Then Lemma \ref{lem2} yields a $*$-isomorphism
$$\psi : M_2(\C)^{\Lambda_1} \rtimes \Lambda_1 \recht M_2(\C)^{\Lambda_2} \rtimes \Lambda_2 \quad\text{with}\quad \psi\Bigl(M_2(\C)^{\Lambda_1}\Bigr) = M_2(\C)^{\Lambda_2} \; .$$
It follows that $\Lambda_1 \cong \Lambda_2$. The converse is obvious.
\end{proof}

\section{Proof of Theorem \ref{thm.A}}\label{sec.proofA}

Theorem \ref{thm.A} will be derived as a consequence of the more general Theorem \ref{thm.tech} below.

\begin{assumptions}\label{assum}
We use, as a black box, the following kind of measure preserving automorphism $T$ of a standard nonatomic probability space $(Y,\mu)$.
\begin{enumerate}
\item $T$ is mixing.
\item The only automorphisms of $(Y,\mu)$ that commute with $T$ are the powers of $T$.
\item The automorphisms $T$ and $T^{-1}$ are not isomorphic: there is no $S \in \Aut(Y,\mu)$ satisfying $STS^{-1} = T^{-1}$.
\item Viewing $T$ as a unitary operator on $L^2(Y,\mu)$, its maximal spectral type is singular w.r.t.\ the Lebesgue measure.
\end{enumerate}
\end{assumptions}

In \cite[Theorems 2.7 and 2.8]{Ru78}, it was shown that Ornstein's rank one automorphisms of \cite{Or70} satisfy conditions 1, 2 and 3. In \cite{Bo91}, these automorphisms were proven to satisfy condition 4 as well. Note that in \cite{Ru78}, conditions 2 and 3 are deduced from a stronger property of $T$~: the mixing automorphism $T$ actually has \emph{minimal self joinings (MSJ)} in the strongest possible sense saying that the only measures on $Y \times Y$ that are invariant under $T^n \times T^m$ (with $n \neq 0$ and $m \neq 0$) and that have marginals $\mu$, are the obvious ones. In later articles, the notion of MSJ has been weakened by only considering $T \times T$ invariant measures with marginals $\mu$. We refer to \cite[Proposition 6.7]{dJR84} for a detailed discussion.

Every automorphism $T \in \Aut(Y,\mu)$ satisfying the assumptions in \ref{assum} gives rise to a mixing pmp action $\Z \actson (Y,\mu)$ that we denote by $(\gamma_n)_{n \in \Z}$ and that has the following properties~: the normalizer of $\Z$ inside $\Aut(Y,\mu)$ equals $\Z$ itself; and there is no nonzero bounded operator $L^2(Y,\mu) \recht \ell^2(\Z)$ that intertwines the unitary representation $\Z \actson L^2(Y,\mu)$ induced by $(\gamma_n)_{n \in \Z}$ with the regular representation of $\Z$.

\begin{theorem}\label{thm.tech}
Let $\Gamma$ be any fixed nonamenable group that contains a copy of $\Z$ as a malnormal\footnote{A subgroup $\Gamma_0 < \Gamma$ is said to be malnormal if $g \Gamma_0 g^{-1} \cap \Gamma_0 = \{e\}$ for all $g \in \Gamma - \Gamma_0$.} subgroup $\Gamma_0$. Fix an automorphism $T \in \Aut(Y,\mu)$ satisfying the assumptions in \ref{assum} and consider the associated action of $\Gamma_0 = \Z$ on $(Y,\mu)$. Put $(X,\mu) = (Y,\mu)^{\Gamma/\Gamma_0}$ and consider the coinduced action $\Gamma \actson (X,\mu)$ as well as the diagonal action $\Z \actson (X,\mu)$. Fix a standard nonatomic finite or infinite measure space $(Z,\eta)$.

Whenever $\Delta \in \Aut(Z,\eta)$ is ergodic and measure preserving, consider
$$N = L^\infty(X \times Z) \rtimes \Z$$
where $\Z$ acts diagonally on $X \times Z$. Consider the action $(\al^\Delta_g)_{g \in \Gamma}$ of $\Gamma$ on $N$ that extends the action $\Gamma \actson (X,\mu)$ and that is the identity on $L^\infty(Z) \rtimes \Z$.

If $\Delta_1,\Delta_2 \in \Aut(Z,\eta)$ are ergodic and measure preserving and if $\psi : N_1 \recht N_2$ is an outer conjugacy between $(\al^{\Delta_1}_g)_{g \in \Gamma}$ and $(\al^{\Delta_2}_g)_{g \in \Gamma}$, there exists a unitary $w \in \cU(N_2)$ and a group element $g_0 \in \Gamma$ such that the outer conjugacy $\psi' = \al^2_{g_0} \circ (\Ad w) \circ \psi$ is the composition of
\begin{itemize}
\item the natural $*$-isomorphism $N_1 \recht N_2$ induced by an automorphism $\theta \in \Aut(Z,\eta)$ satisfying $\theta \circ \Delta_1 \circ \theta^{-1} = \Delta_2$.
\item the natural automorphism of $N_2$ induced by an automorphism $\delta \in \Aut(\Gamma)$ satisfying $\delta(g) = g$ for all $g \in \Gamma_0$,
\end{itemize}
\end{theorem}

We say that actions $(\al^i_g)_{g \in \Gamma}$, $i=1,2$, of a group $\Gamma$ on von Neumann algebras $N_i$ are \emph{isomorphic} if there exists a $*$-isomorphism $\psi : N_1 \recht N_2$ such that $\al^2_g \circ \psi = \psi \circ \al^1_g$ for all $g \in \Gamma$. We also use the notations $\module(\psi)$ and $\module(\theta)$ to denote the scaling factor of a trace scaling automorphism $\psi \in \Aut(N)$, or a measure scaling automorphism $\theta \in \Aut(Z,\eta)$.

Within the context of Theorem \ref{thm.tech}, we then also have the following results.
\begin{align*}
& \{\module(\psi) \mid \psi \in \Aut(N) \;\;\text{is an outer conjugacy of}\;\; \al^\Delta\}
\\ = \; & \{\module(\psi) \mid \psi \in \Aut(N) \;\;\text{commutes with}\;\; \al^\Delta\}
\\ = \; & \{\module(\theta) \mid \theta \in \Aut(Z,\eta) \;\;\text{commutes with}\;\; \Delta\} \; .
\end{align*}

If $\Delta_1,\Delta_2 \in \Aut(Z,\eta)$ are ergodic and measure preserving, then the following statements are equivalent.
\begin{itemize}
\item $\Delta_1$ is conjugate with $\Delta_2$~: there exists a nonsingular automorphism $\theta \in \Aut(Z,\eta)$ such that $\Delta_2 = \theta \circ \Delta_1 \circ \theta^{-1}$ a.e.
\item The actions $(\al^{\Delta_1}_g)_{g \in \Gamma}$ and $(\al^{\Delta_2}_g)_{g \in \Gamma}$ are isomorphic.
\item The actions $(\al^{\Delta_1}_g)_{g \in \Gamma}$ and $(\al^{\Delta_2}_g)_{g \in \Gamma}$ are outer conjugate.
\end{itemize}

\begin{proof}[Proof of Theorem \ref{thm.tech}]
We start by making the notations compatible with those at the beginning of Section \ref{sec.outer-conj}. We denote $B = L^\infty(Y,\mu)$ and $A = B^{\Gamma/\Gamma_0} = L^\infty(X,\mu)$. For every $g \Gamma_0 \in \Gamma/\Gamma_0$, we denote by $\pi_{g \Gamma_0} : B \recht A$ the embedding of $B$ as the $g\Gamma_0$-th tensor factor of $A = B^{\Gamma/\Gamma_0}$. We have $\Gamma_0 = \Z = \Lambda$ and the actions $(\beta_g)_{g \in \Gamma_0}$ and $(\gamma_g)_{g \in \Gamma_0}$ on $B$ are equal and both induced by $T$.

We write $D = L^\infty(Z,\eta)$. The ergodic measure preserving automorphisms $\Delta_1,\Delta_2 \in \Aut(Z,\eta)$ induce essentially free and ergodic actions $(\gamma^i_n)_{n \in \Z}$ of $\Lambda = \Z$ on $D$. We consider $N_1,N_2$ as in the formulation of the Theorem, but we denote the actions by $(\al^i_g)_{g \in \Gamma}$ rather than $(\al^{\Delta_i}_g)_{g \in \Gamma}$.
Assume that $\psi : N_1 \recht N_2$ is an outer conjugacy between $(\al^1_g)_{g \in \Gamma}$ and $(\al^2_g)_{g \in \Gamma}$.

By Lemma \ref{lem2}, and after replacing $\psi$ by $(\Ad w) \circ \psi$, we may assume that
\begin{equation}\label{eq.ok}
\psi(D \rtimes_{\gamma^1} \Z) = D \rtimes_{\gamma^2} \Z \quad , \quad \psi(D) = D \quad\text{and}\quad \psi \circ \al^1_g = \al^2_{\delta(g)} \circ \psi
\end{equation}
for all $g \in \Gamma$ and some automorphism $\delta \in \Aut(\Gamma)$. Taking the relative commutant of $\psi(D) = D$, we also have that $\psi(A \ovt D) = A \ovt D$.

We prove now the existence of a $g_0 \in \Gamma$ such that $g_0 \delta(\Gamma_0) g_0^{-1} \cap \Gamma_0 \neq \{e\}$. If such a $g_0$ does not exist, it follows from Lemma \ref{lem.regular} that the unitary representation $(\al^2_{\delta(g)})_{g \in \Gamma_0}$ on $L^2(N_2) \ominus L^2(D \rtimes_{\gamma^2} \Z)$ is a multiple of the regular representation of $\Gamma_0$. On the other hand, by condition 4 in \ref{assum}, the unitary representation $\Gamma_0 \actson L^2(\pi_{e \Gamma_0}(B) \ot D)$ given by $(\al^1_g)_{g \in \Gamma_0}$ is disjoint from the regular representation of $\Gamma_0$. Combining both observations, it follows that $\psi(\pi_{e \Gamma_0}(B) \ot D) \subset D \rtimes_{\gamma^2} \Z$. Since $\psi \circ \al^1_g = \al^2_{\delta(g)} \circ \psi$ for all $g \in \Gamma$ and using \eqref{eq.ok}, we arrive at the contradiction that $\psi(N_1) \subset D \rtimes_{\gamma^2} \Z$. So there indeed exists a $g_0 \in \Gamma$ such that $g_0 \delta(\Gamma_0) g_0^{-1} \cap \Gamma_0 \neq \{e\}$.

After replacing $\psi$ by $\al^2_{g_0} \circ \psi$ and $\delta$ by $(\Ad g_0) \circ \delta$, we may assume that $\delta(\Gamma_0) \cap \Gamma_0 \neq \{e\}$ and we find $g_1 \in \Gamma_0$ with $g_1 \neq e$ and $\delta(g_1) \in \Gamma_0$. Since $\Gamma_0$ is abelian, it follows that $\delta^{-1}(\Gamma_0)$ commutes with $g_1$. By malnormality of $\Gamma_0 < \Gamma$, we conclude that $\delta^{-1}(\Gamma_0) \subset \Gamma_0$. Applying $\delta$, it follows that $\Gamma_0 \subset \delta(\Gamma_0)$. Since $\delta(\Gamma_0)$ is abelian and $\Gamma_0 < \Gamma$ is malnormal, we find that $\Gamma_0 = \delta(\Gamma_0)$.

For $i=1,2$, define $K_i = L^2\bigl((\pi_{e \Gamma_0}(B) \ovt D) \rtimes_{\gamma^i} \Z\bigr)$ and $L_i = L^2(N_i) \ominus K_i$. Since $\Gamma_0 < \Gamma$ is malnormal, it follows from Lemma \ref{lem.regular} that the unitary representation $(\al^2_{\delta(g)})_{g \in \Gamma_0}$ of $\Gamma_0$ on $L_2$ is a multiple of the regular representation of $\Gamma_0$. By condition~4 in \ref{assum}, the unitary representation $(\al^1_g)_{g \in \Gamma_0}$ of $\Gamma_0$ on $K_1$ is disjoint from the regular representation of $\Gamma_0$. Combining both observations, it follows that $\psi(K_1) \subset K_2$. By symmetry, also the converse inclusion holds. So
$$\psi\bigl( (\pi_{e \Gamma_0}(B) \ovt D) \rtimes_{\gamma^1} \Z \bigr) = (\pi_{e \Gamma_0}(B) \ovt D) \rtimes_{\gamma^2} \Z \; .$$
We therefore find the $*$-isomorphism
$$\Psi : (B \ovt D) \rtimes_{\gamma^1} \Z \recht (B \ovt D) \rtimes_{\gamma^2} \Z$$
satisfying $(\pi_{e\Gamma_0} \ot \id) \circ \Psi = \psi \circ (\pi_{e\Gamma_0} \ot \id)$. Because of \eqref{eq.ok} and the facts that $\psi(A \ovt D) = A \ovt D$ and $\psi(D) = D$, we have
$$\Psi(B \ovt D) = B \ovt D \quad , \quad \Psi(D \rtimes_{\gamma^1} \Z) = D \rtimes_{\gamma^2} \Z \quad\text{and}\quad \Psi(D) = D \; .$$

Since $D = L^\infty(Z,\eta)$, we obtain the nonsingular automorphism $\theta \in \Aut(Z,\eta)$ satisfying $\Psi(d) = d \circ \theta^{-1}$ for all $d \in L^\infty(Z,\eta)$. Since
$\Psi(D \rtimes_{\gamma^1} \Z) = D \rtimes_{\gamma^2} \Z$, the map $\theta$ is an orbit equivalence between the essentially free actions $\gamma^i : \Z \actson (Z,\eta)$. So we find the $1$-cocycle $\om : \Z \times Z \recht \Z$ satisfying $\theta(n \cdot z) = \om(n,z) \cdot \theta(z)$ for all $n \in \Z$ and a.e.\ $z \in Z$. For all $n,m \in \Z$, denote by $p_{n,m} \in L^\infty(Z)$ the projection with support $\{z \in Z \mid \om(n,\theta^{-1}(z)) = m \}$.
Denote by $(u_n)_{n \in \Z}$ the canonical unitaries in the crossed product $D \rtimes_{\gamma^i} \Z$. Then,
\begin{equation}\label{eq.formula-us}
\Psi(u_n) = \sum_{m \in \Z} u_m \, \mu_{n,m}
\end{equation}
for all $n \in \Z$, where $\mu_{n,m} \in \cU(D p_{n,m})$.

We identify $B \ovt D = L^\infty(Z,B)$. Since the automorphism $\Psi \in \Aut(B \ovt D)$ satisfies $\Psi(D) = D$, we find a measurable family of automorphisms $\psi_z \in \Aut(B)$, for a.e.\ $z \in Z$, such that
$$(\Psi(b))(\theta(z)) = \psi_z(b(z)) \quad\text{for all}\;\; b \in L^\infty(Z,B) \;\;\text{and a.e.}\; z \in Z \; .$$
The $*$-isomorphism $\psi$ scales the trace by a scaling factor $\lambda > 0$. Thus $\theta$ scales the measure by the same factor $\lambda$ and for a.e.\ $z \in Z$, the automorphism $\psi_z$ is trace preserving.

For all $n \in \Z$ and $b \in B$, we have
$$\Psi(\gamma_{-n}(b) \ot 1) = \Psi(u_n^* (b \ot 1) u_n) = \Psi(u_n)^* \, \Psi(b \ot 1) \, \Psi(u_n) \; .$$
The left and right hand side both belong to $L^\infty(Z,B)$. Evaluating the left and right hand side in a point $\theta(z)$ for some $z \in Z$ with $\om(n,z) = m$ and using \eqref{eq.formula-us}, we obtain the equality
$$\psi_z(\gamma_{-n}(b)) = \gamma_{-m}(\psi_{n \cdot z}(b)) \; .$$
It follows that
\begin{equation}\label{eq.commutation-rel}
\psi_{n \cdot z} \circ \gamma_n = \gamma_{\om(n,z)} \circ \psi_z
\end{equation}
for all $n \in \Z$ and a.e.\ $z \in Z$.

Since $\psi \circ \al^1_g = \al^2_{\delta(g)} \circ \psi$ for all $g \in \Gamma$, it follows that $\Psi \circ \beta_g = \beta_{\delta(g)} \circ \Psi$ for all $g \in \Gamma_0$. This means that $\psi_z \circ \beta_g = \beta_{\delta(g)} \circ \psi_z$ for all $g \in \Gamma_0$ and a.e.\ $z \in Z$. By conditions 2 and 3 in \ref{assum}, it follows that $\delta(g) = g$ for all $g \in \Gamma_0$ and that a.e.\ $\psi_z$ is given by a power of $T$. So we find a measurable map $\vphi : Z \recht \Z$ such that $\psi_z = \gamma_{\vphi(z)}$ for a.e.\ $z \in Z$.

Writing $\om'(n,z) = \vphi(n \cdot z)^{-1} \, \om(n,z) \, \vphi(z)$, it then follows from \eqref{eq.commutation-rel} that $\gamma_{\om'(n,z)} = \gamma_n$. We conclude that $\om'(n,z) = n$ for all $n \in \Z$ and a.e.\ $z \in Z$. It follows that the map $z \mapsto \vphi(z)^{-1} \cdot \theta(z)$ is an automorphism of $(Z,\eta)$. Denoting, for all $m \in \Z$, by $q_m \in L^\infty(Z)$ the projection with support $\{z \in Z \mid \vphi(\theta^{-1}(z)) = m\}$, it follows that
$$U = \sum_{m \in \Z} q_m u_m$$
is a well defined unitary operator in $D \rtimes_{\gamma^2} \Z$. Replacing $\psi$ by $(\Ad U^*) \circ \psi$, we get that
$$\psi((\pi_{e \Gamma_0}(b) \ot d)u_n) = (\pi_{e \Gamma_0}(b) \ot \theta(d))u_n$$
for all $b \in B$, $d \in D$, $n \in \Z$, where $\theta \in \Aut(D)$ satisfies $\theta \circ \gamma^1_n = \gamma^2_n \circ \theta$ for all $n \in \Z$. We still have that
$\psi \circ \al^1_g = \al^2_{\delta(g)} \circ \psi$ for all $g \in \Gamma$, where $\delta \in \Aut(\Gamma)$ satisfies $\delta(g) = g$ for all $g \in \Gamma_0$.

So $\psi$ is indeed the composition of the two isomorphism described in the statement of the theorem.
\end{proof}

\begin{remark}
Fix the same actions $\Gamma \actson (X,\mu)$ and $\Z \actson (X,\mu)$ as in Theorem \ref{thm.tech}. Whenever $(\eta_n)_{n \in \Z}$ is an outer action of $\Z$ on the hyperfinite II$_1$ factor $R$, we consider, in the same way as in Theorem \ref{thm.tech}, the action $(\al^\eta_g)_{g \in \Gamma}$ of $\Gamma$ on
$$(L^\infty(X) \ovt R) \rtimes_\eta \Z \; .$$
Contrary to the situation in Theorem \ref{thm.tech}, where we use the abelian algebra $L^\infty(Z)$ instead of $R$, this construction is of little interest since all the actions $(\al^\eta_g)_{g \in \Gamma}$ are isomorphic. Indeed, take two outer actions $(\eta_n)_{n \in \Z}$ and $(\eta'_n)_{n \in \Z}$ of $\Z$ on $R$. By \cite[Theorem 2]{Co75}, the automorphisms $\eta_1$ and $\eta'_1$ are outer conjugate. So we find $\psi_0 \in \Aut(R)$ and a unitary $v_1 \in \cU(R)$ such that $\psi_0 \circ \eta_1 = \Ad v_1 \circ \eta'_1 \circ \psi_0$. Denoting by $(u_n)_{n \in \Z}$ the canonical generating unitaries of $L(\Z)$, one checks that there is a unique $*$-isomorphism
$$\psi : (L^\infty(X) \ovt R) \rtimes_\eta \Z \recht (L^\infty(X) \ovt R) \rtimes_{\eta'} \Z$$
satisfying $\psi(a \ot b) = a \ot \psi_0(b)$ for all $a \in L^\infty(X)$, $b \in R$ and $\psi(u_1) = (1 \ot v_1) u_1$. By construction, $\psi \circ \al^\eta_g = \al^{\eta'}_g \circ \psi$ for all $g \in \Gamma$.
\end{remark}

We are now ready to prove Theorem \ref{thm.A}.

\begin{proof}[Proof of Theorem \ref{thm.A}]
We apply Theorem \ref{thm.tech} to the group $\Gamma = (\Z / n \Z) * (\Z / m\Z)$. Denote by $a \in \Z/n \Z$ and $b \in \Z / m\Z$ the cyclic generators. Define $\Gamma_0 \cong \Z$ as the subgroup of $\Gamma$ generated by $ab$. By \cite[Example 7.C]{dlHW11}, $\Gamma_0$ is a malnormal subgroup of $\Gamma$. For every ergodic measure preserving automorphism $\Delta$ of a standard nonatomic finite or infinite measure space $(Z,\eta)$, Theorem \ref{thm.tech} provides the outer action $(\al^\Delta_g)_{g \in \Gamma}$ of $\Gamma$ on $N = L^\infty(X \times Z) \rtimes \Z$.

To prove the first statement of Theorem \ref{thm.A}, take $H \in \cS$. By \cite[Theorem 4.3]{Aa86}, there exists an ergodic measure preserving automorphism $\Delta$ of a standard infinite measure space $(Z,\eta)$ such that
\begin{equation}\label{eq.hulp}
H = \{\module(\theta) \mid \theta \in \Aut(Z,\eta) \;\;\text{commutes with}\;\; \Delta \} \; .
\end{equation}
Fix a nonzero projection $p \in L^\infty(Z)$ of finite trace and realize the hyperfinite II$_1$ factor $R$ as $pNp$. We still denote by $\al^\Delta$ the restriction of $\al^\Delta$ to $R = pNp$. We consider the associated group-type subfactor $S(\al^\Delta)$. We claim that $\cF(S(\al^\Delta)) = H$.

First assume that $\lambda \in \cF(S(\al^\Delta))$. By \cite[Theorem 3.2]{BNP06}, we find a projection $q \in L^\infty(Z)$ with $\Tr(q) = \lambda \Tr(p)$ and a $*$-isomorphism $\psi : pNp \recht qNq$ such that $\psi$ is an outer conjugacy between the restrictions of $\al^\Delta$ to $pNp$, resp.\ $qNq$. We can then amplify $\psi$ to an outer conjugacy $\psi \in \Aut(N)$ of $\al^\Delta$ scaling the trace by the module $\lambda$. Combining the remarks after Theorem \ref{thm.tech} with formula \eqref{eq.hulp}, we conclude that $\lambda \in H$.

Conversely, if $\lambda \in H$, we find by \eqref{eq.hulp} an automorphism $\psi \in \Aut(N)$ that commutes with the action $\al^\Delta$ and that scales the trace by the module $\lambda$. Put $q = \psi(p)$. Then $q$ is an $\al^\Delta$-invariant projection in $N$ with $\Tr(q) = \lambda \Tr(p)$ and $\psi$ induces an isomorphism between the subfactors
$$(pNp)^{\Z/n\Z} \subset (pNp) \rtimes \Z/m\Z \qquad\text{and}\qquad (qNq)^{\Z / n \Z} \subset (qNq) \rtimes \Z/m\Z \; .$$
This precisely means that $\lambda \in \cF(S(\al^\Delta))$.

To prove the second statement of Theorem \ref{thm.A}, we take for $(Z,\eta)$ a standard nonatomic probability space. For every ergodic pmp automorphism $\Delta \in \Aut(Z,\eta)$, Theorem
\ref{thm.tech} provides an outer action $(\al^\Delta_g)_{g \in \Gamma}$ on the hyperfinite II$_1$ factor $R = L^\infty(X \times Z) \rtimes \Z$. We denote by $S(\al^\Delta)$ the associated group-type subfactor.

If the subfactors $S(\al^{\Delta_1})$ and $S(\al^{\Delta_2})$ are isomorphic, it follows from \cite[Theorem 3.2]{BNP06} that the actions $\al^{\Delta_1}$ and $\al^{\Delta_2}$ are outer conjugate. It then follows from Theorem \ref{thm.tech} that $\Delta_1,\Delta_2$ are conjugate inside $\Aut(Z,\eta)$. Conversely, when $\Delta_1,\Delta_2$ are conjugate inside $\Aut(Z,\eta)$, the actions $\al^{\Delta_1}$, $\al^{\Delta_2}$ are isomorphic and hence, the subfactors $S(\al^{\Delta_1})$, $S(\al^{\Delta_2})$ are isomorphic.

Finally note that by \cite[Theorem 5.1]{BDG08}, the standard invariant of the subfactor $R^H \subset R \rtimes K$ only depends on the inclusions $H,K \subset \Gamma$.
\end{proof}

As explained in the introduction, we also provide as a corollary of Theorem \ref{thm.tech}, the following new explicit construction of II$_1$ factors with a prescribed fundamental group in the family $\cS$ of \cite[Section 2]{PV08}.

\begin{corollary} \label{cor.fund-group-II1}
Let $\Gamma$ be a nonamenable, weakly amenable, bi-exact, icc group containing a copy of $\Z$ as a malnormal subgroup $\Gamma_0 < \Gamma$, e.g.\ take $\Gamma=\F_2 = \Z * \Z$ with $\Gamma_0$ given by the first copy of $\Z$. For every ergodic measure preserving automorphism $\Delta$ of a standard infinite measure space $(Z,\eta)$, consider the action $(\al^\Delta_g)_{g \in \Gamma}$ of $\Gamma$ on $N = L^\infty(X \times Z) \rtimes \Z$ as in Theorem \ref{thm.tech}. Fix a projection $p \in L^\infty(Z)$ with $\Tr(p) < \infty$.

The fundamental group of the II$_1$ factor $pNp \rtimes \Gamma$ is given by $\module\bigl(\Centr_{\Aut(Z,\eta)}(\Delta)\bigr)$.
\end{corollary}
\begin{proof}
First assume that $\lambda > 0$ belongs to the fundamental group of $pNp \rtimes \Gamma$. We can then take a projection $q \in L^\infty(Z)$ with $\Tr(q) = \lambda \Tr(p)$ and a $*$-isomorphism $\psi : pNp \rtimes \Gamma \recht qNq \rtimes \Gamma$. By \cite[Theorem 1.4]{PV12}, we have $\psi(pNp) \prec qNq$ and $qNq \prec \psi(pNp)$. It then follows from \cite[Lemma 8.4]{IPP05} that $\psi(pNp)$ and $qNq$ are unitarily conjugate. So $\psi$ defines a cocycle conjugacy of the action $(\al^\Delta_g)_{g \in \Gamma}$ scaling the trace by the module $\lambda$. By Theorem \ref{thm.tech}, we have that $\lambda = \module(\theta)$ for some $\theta \in \Aut(Z,\eta)$ that commutes with $\Delta$.

Conversely, every $\theta \in \Aut(Z,\eta)$ that commutes with $\Delta$ defines an automorphism of $N$ that commutes with $\al^\Delta$ and hence extends to an automorphism of $N \rtimes \Gamma$ scaling the trace by the module $\lambda$. It follows that $\module(\theta)$ belongs to the fundamental group of $pNp \rtimes \Gamma$.
\end{proof}

\end{document}